\documentclass[11pt]{amsart}

\author[P.~Miska]{Piotr Miska}
\address{Jagiellonian University, Faculty of Mathematics and Computer Science, Krak\'ow, Poland}
\email{piotrmiska91@gmail.com}

\author[N.~Murru]{Nadir Murru}
\address{Universit\`a degli Studi di Torino, Department of Mathematics, Torino, Italy}
\email{nadir.murru@unito.it}
\urladdr{\url{http://orcid.org/0000-0003-0509-6278}}

\author[C.~Sanna]{Carlo Sanna}
\address{Universit\`a degli Studi di Torino, Department of Mathematics, Torino, Italy}
\email{carlo.sanna.dev@gmail.com}
\urladdr{\url{http://orcid.org/0000-0002-2111-7596}}

\keywords{denseness, $p$-adic numbers, polynomials, quotient set, sum of powers}
\subjclass[2010]{Primary: 11B05; Secondary: 11B83}

\title{On the $p$-adic denseness of the quotient set of a polynomial image}

\usepackage{amsmath}
\usepackage{amssymb}
\usepackage{amsthm}
\usepackage{geometry}
\usepackage{hyperref}
\usepackage{enumerate}
\hypersetup{colorlinks=true}
\usepackage[utf8]{inputenc}

\newtheorem{thm}{Theorem}[section]
\newtheorem{lem}[thm]{Lemma}
\newtheorem{cor}[thm]{Corollary}
\newtheorem{prop}[thm]{Proposition}
\newtheorem{que}[thm]{Question}
\theoremstyle{remark}
\newtheorem{rmk}[thm]{Remark}
\newtheorem{exe}[thm]{Example}

\begin{document}

\begin{abstract}
The \emph{quotient set}, or \emph{ratio set}, of a set of integers $A$ is defined as 
\begin{equation*}
R(A) := \left\{a/b : a,b \in A,\; b \neq 0\right\} .
\end{equation*}
We consider the case in which $A$ is the image of $\mathbb{Z}^+$ under a polynomial $f \in \mathbb{Z}[X]$, and we give some conditions under which $R(A)$ is dense in $\mathbb{Q}_p$. 
Then, we apply these results to determine when $R(S_m^n)$ is dense in $\mathbb{Q}_p$, where $S_m^n$ is the set of numbers of the form $x_1^n + \cdots + x_m^n$, with $x_1, \dots, x_m \geq 0$ integers. 
This allows us to answer a question posed in [Garcia \textit{et~al.}, $p$-adic quotient sets, Acta Arith.~\textbf{179}, 163--184].
We end leaving an open question.
\end{abstract}

\maketitle

\section{Introduction}

The \emph{quotient set}, also known as \emph{ratio set}, of a set of integers $A$ is defined as
\begin{equation*}
R(A) := \left\{ \frac{a}{b} : a, b \in A, \; b \neq 0 \right\} .
\end{equation*}
The question of when $R(A)$ is dense in $\mathbb{R}^+$ is a classical topic and has been studied by many researchers (see, e.g., \cite{MR3229105, MR1475512, MR1390582, MR2843990, MR2549381, MR1197643, MR0242756, MR1659159}).

Recently, some authors approached the study of the denseness of $R(A)$ in the field of \mbox{$p$-adic} numbers $\mathbb Q_p$. 
Garcia and Luca~\cite{MR3593645} proved that the quotient set of the Fibonacci numbers is dense in $\mathbb{Q}_p$, and Sanna~\cite{MR3668396} extended this result to the $k$-generalized Fibonacci numbers. 
In~\cite{MR3670202}, the denseness of $R(A)$ in $\mathbb{Q}_p$ is studied when $A$ is the set of values of a Lucas sequence, the set of positive integers which are sum of $k$ squares, respectively $k$ cubes, or the union of two geometric progressions.
Moreover, Miska and Sanna~\cite{MS18} proved that, given any partition $A_1, \dots, A_k$ of $\mathbb{Z}^+$, for all prime numbers $p$ but at most $\lfloor \log_2 k\rfloor$ exceptions at least one of $R(A_1), \dots, R(A_k)$ is dense in $\mathbb{Q}_p$.

In this paper, we focus on the study of the denseness of $R(A)$ in $\mathbb{Q}_p$ when $A$ is the image of $\mathbb{Z}^+$ under a polynomial $f \in \mathbb{Z}[X]$.  
For the sake of notation, we put $R_f := R(f(\mathbb{Z}^+))$ for any function $f : \mathbb{Z} \to \mathbb{Q}_p$.
The following easy lemma provides a necessary condition under which $R_f$ is dense in $\mathbb{Q}_p$.

\begin{lem}\label{lem:nec-dens-cond}
Let $f : \mathbb Z_p \rightarrow \mathbb Q_p$ be a continuous function. If $R_f$ is dense in $\mathbb Q_p$, then $f$ has a zero in $\mathbb Z_p$.
\end{lem}
\begin{proof}
Since $R_f$ is dense in $\mathbb{Q}_p$, there exists a sequence of integers $(x_n)_{n \geq 0}$ such that $f(x_n) \to 0$ (in the $p$-adic topology) as $n \to \infty$.
By the compactness of $\mathbb{Z}_p$, there exists a subsequence $(x_{n_k})_{k \geq 0}$ converging to some $x_\infty \in \mathbb{Z}_p$. 
Since $f$ is continuous, we get $f(x_\infty) = 0$, as desired. 
\end{proof}

Our first result is a sufficient condition under which $R_f$ is dense in $\mathbb{Q}_p$.
We postpone its proof to Section~\ref{sec:2}.

\begin{thm}\label{thm:dens-cond}
Let $f : \mathbb Z_p \rightarrow \mathbb Q_p$ be an analytic function and let $z_1, z_2 \in \mathbb{Z}_p$ be two (not necessarily distinct) zeros of $f$ of multiplicities $\mu_1, \mu_2$, respectively. 
If $\mu_1,\mu_2$ are coprime, then $R_f$ is dense in $\mathbb Q_p$.
\end{thm}

As an immediate consequence we have the following corollary.

\begin{cor}\label{cor:dens-cond}
If $f : \mathbb{Z}_p \to \mathbb{Q}_p$ is an analytic function with a simple zero in $\mathbb Z_p$, then $R_f$ is dense in $\mathbb{Q}_p$.
\end{cor}

The above results make possible to completely characterize the linear and quadratic polynomials $f$ for which $R_f$ is dense in $\mathbb{Q}_p$.

\begin{prop}
Let $f \in \mathbb Z[X]$ be a polynomial of degree $1$ or $2$.
Then, $R_f$ is dense in $\mathbb Q_p$ if and only if $f$ has a simple zero in $\mathbb{Z}_p$.
\end{prop}
\begin{proof}
When $f$ has degree $1$, the thesis follows immediately from Lemma~\ref{lem:nec-dens-cond} and Corollary~\ref{cor:dens-cond}. 
Assume $f$ has degree $2$.
If $f$ has a simple zero in $\mathbb{Z}_p$, then $R_f$ is dense in $\mathbb{Q}_p$ by Corollary~\ref{cor:dens-cond}. 
On the other hand, if $f$ has no simple zeros in $\mathbb{Z}_p$, then we have two cases. 
In the first case, $f$ has no zeros in $\mathbb{Z}_p$.
Then, by Lemma~\ref{lem:nec-dens-cond}, $R_f$ is not dense in $\mathbb{Q}_p$. 
In the second case, $f$ has a zero in $\mathbb{Z}_p$ with multiplicity $2$, i.e., $f(x) = a(x-z)^2$, for some $a,z \in \mathbb{Z}_p$ with $a \neq 0$.
Consequently, $R_f$ is not dense in $\mathbb{Q}_p$, since the $p$-adic valuation of each element of $R_f$ is divisible by $2$.
\end{proof}

For polynomials of higher degrees, we can not exploit Lemma~\ref{lem:nec-dens-cond} and Corollary~\ref{cor:dens-cond} to determine if $R_f$ is dense in $\mathbb{Q}_p$.
For instance, consider the case of a polynomial of degree $3$ with a double root in $\mathbb Z_p$ and the other root not in $\mathbb{Z}_p$.
However, if we consider polynomials having all their roots in $\mathbb{Z}_p$, then we have the following result.

\begin{prop}\label{prop:31}
Let $f \in \mathbb Z[X]$ be a nonconstant polynomial splitting in $\mathbb{Z}_p$ and of degree less than $31$.
Then, $R_f$ is not dense in $\mathbb{Q}_p$ if and only if there exists an integer $n > 1$ which divides the multiplicity of each root of $f$.
\end{prop}
\begin{proof}
Let $\mu_1, \dots, \mu_s$ be the multiplicities of the roots of $f$.
If there exists an integer $n > 1$ dividing all $\mu_1, \dots, \mu_s$, then $f = ag^n$, for some $a \in \mathbb{Z}\setminus\{0\}$ and $g \in \mathbb{Z}[X]$.
Consequently, $R_f$ is not dense in $\mathbb{Q}_p$, since the $p$-adic valuation of each element of $R_f$ is divisible by $n$.
Now suppose that there exists no integer $n > 1$ dividing all $\mu_1, \dots, \mu_s$.
We shall prove that $\gcd(\mu_i, \mu_j) = 1$ for some $i,j$. In this way, by Theorem~\ref{thm:dens-cond}, it follows that $R_f$ is dense in $\mathbb{Q}_p$.
For the sake of contradiction, assume $\gcd(\mu_i, \mu_j) > 1$ for all $i,j$.
In particular, we have $s \geq 3$, and that each $\mu_i$ has at least two distinct prime factors.
Also, at least one of $\mu_1, \ldots, \mu_s$ is odd.
Without loss of generality, we can assume $\mu_1$ odd.
Thus $\mu_1 \in \{15, 21\}$, and at least one of $\mu_2, \ldots, \mu_s$ is not divisible by $3$.
Without loss of generality, we can assume $\mu_2$ not divisible by $3$.
Thus $\mu_2 \in \{10, 14\}$.
Since $\mu_3$ has at least two distinct prime factors, $\mu_3 \geq 6$ and consequently $\deg f = \mu_1 + \cdots + \mu_s > 30$, absurd. 
\end{proof}

\begin{rmk}
Proposition~\ref{prop:31} is optimal in the sense that there exists a polynomial $f \in \mathbb{Z}[X]$ of degree $31$, splitting in $\mathbb{Z}_p$, with the greatest common divisor of the multiplicities of its roots equal to $1$, but such that $R_f$ is not dense in $\mathbb{Q}_p$.
Indeed, consider 
\begin{equation*}
f(X)= (X + 1)^6 (X + 2)^{10}(X + 3)^{15} .
\end{equation*}
Then, for $p > 2$ (respectively $p = 2$) the $p$-adic valuation of each element of $f(\mathbb{Z}^+)$ is of the form $6n$, $10n$, or $15n$ (respectively $10n$, $6n+15$, or $15n + 6$), for some integer $n \geq 0$.
Therefore, no element of $R_f$ has $p$-adic valuation equal to $1$ (respectively $2$), and $R_f$ is not dense in $\mathbb{Q}_p$. 
\end{rmk}

\begin{rmk}
Using the same reasonings as in the proof of Proposition~\ref{prop:31}, one can prove a slightly more general statement: 
Given $f = gh$, where $g,h \in \mathbb{Z}[X]$ are such that $g$ splits in $\mathbb{Z}_p$, $1 \leq \deg g \leq 30$, and the $p$-adic valuation of $h$ is constant, we have that $R_f$ is not dense in $\mathbb{Q}_p$ if and only if there does not exist an integer $n > 1$ dividing all the multiplicities of the roots of $g$. 
\end{rmk}

For integers $m,n \geq 2$, define the set
\begin{equation*}
S_m^n := \left\{x_1^n + \cdots + x_m^n : x_1, \ldots,x_m \in \mathbb{Z}_{\geq 0}\right\} .
\end{equation*}
The authors of~\cite{MR3670202} considered $n = 2,3$ and proved the following results~\cite[Theorems~4.1 and~4.2]{MR3670202}.
(Actually, there is a small error, here corrected, in \cite[Theorem~4.2]{MR3670202}, see Remark~\ref{rmk:mistake} below.)

\begin{thm}\label{thm:squares-cubes}
For all prime numbers $p$, we have:
\begin{enumerate}[\phantom{m}(a)]
\setlength{\itemsep}{0.5em}
\item $R(S_2^2)$ is dense in $\mathbb{Q}_p$ if and only if $p \equiv 1 \pmod 4$.
\item $R(S_m^2)$ is dense in $\mathbb{Q}_p$ for all integers $m \geq 3$.
\item $R(S_m^3)$ is dense in $\mathbb{Q}_p$ for all integers $m \geq 2$.
\end{enumerate}
\end{thm}

For all integers $n, b \geq 2$, let $\gamma(n, b)$ denote the smallest positive integer $g$ such that for every $a \in \mathbb{Z}$ the equation
\begin{equation}\label{equ:Xg}
X_1^n + \cdots + X_g^n \equiv a \pmod b
\end{equation}
has a solution.
Furthermore, let $\theta(n, b)$ be the smallest positive integer $g$ such that for $a = 0$ the equation \eqref{equ:Xg} has a solution with at least one of $X_1, \dots, X_g$ coprime with $b$.
The quantities $\gamma(n, b)$, $\theta(n, b)$ have been studied in regard to analogs of Waring's problem modulo~$p$ (see, e.g., \cite{MR0439787,MR0424734}).

We give an effective criterion to establish if $R(S_m^n)$ is dense in $\mathbb{Q}_p$.
We postpone its proof to Section~\ref{sec:3}.

\begin{thm}\label{thm:power-sums}
Let $m, n \geq 2$ be integers, let $p$ be a prime number, and put $k := \nu_p(n)$.
\begin{enumerate}[\phantom{m}(a)]
\setlength{\itemsep}{0.5em}
\item If $m \geq \theta(n, p^{2k + 1})$, then $R(S_m^n)$ is dense in $\mathbb{Q}_p$.
\item If $m < \theta(n, p^{2k + 1})$ and $(n, p) \notin \{(2, 2), (4, 2), (8, 2), (16, 2)\}$, then $R(S_m^n)$ is not dense in~$\mathbb{Q}_p$.
\item $R(S_m^2)$ is dense in $\mathbb{Q}_2$ if and only if $m \geq 3$.
\item $R(S_m^4)$ is dense in $\mathbb{Q}_2$ if and only if $m \geq 8$.
\item $R(S_m^8)$ is dense in $\mathbb{Q}_2$ if and only if $m \geq 16$.
\item $R(S_m^{16})$ is dense in $\mathbb{Q}_2$ if and only if $m \geq 64$.
\end{enumerate}
\end{thm}

\begin{exe}
Let us consider the denseness of $R(S_m^6)$ in $\mathbb{Q}_{11}$.
In order to apply Theorem~\ref{thm:power-sums}, we have to compute $\theta(6, 11)$.
The nonzero sixth powers modulo $11$ are $1$, $3$, $4$, $5$, and $9$.
Hence, the minimum positive integer $g$ such that the equation $X_1^6 + \cdots + X_g^6 \equiv 0 \pmod {11}$ has a solution, with at least one of $X_1, \dots, X_g$ not divisible by $11$, is $\theta(6, 11) = 3$.
Consequently, by points (a) and (b) of Theorem~\ref{thm:power-sums}, we have that $R(S_m^6)$ is dense in $\mathbb{Q}_{11}$ if and only if $m \geq 3$.
\end{exe}

\begin{exe}
Let us consider the denseness of $R(S_m^{10})$ in $\mathbb{Q}_2$.
In order to apply Theorem~\ref{thm:power-sums}, we have to compute $\theta(10, 8)$.
We have $x^{10} \equiv 1 \pmod 8$ for each odd integer $x$.
Hence, it follows easily that $\theta(10, 8) = 8$.
Consequently, by points (a) and (b) of Theorem~\ref{thm:power-sums}, we have that $R(S_m^{10})$ is dense in $\mathbb{Q}_2$ if and only if $m \geq 8$.
\end{exe}

For $m = 2$, we have the following corollary.

\begin{cor}\label{cor:minusone}
Let $n \geq 2$ be an integer, let $p$ be a prime number, and put $k = \nu_p(n)$.
Then $R(S_2^n)$ is dense in $\mathbb{Q}_p$ if and only if $-1$ is an $n$th power modulo $p^{2k + 1}$.
In particular, $R(S_2^n)$ is dense in $\mathbb{Q}_p$ whenever $n$ is odd.
\end{cor}
\begin{proof}
First, assume $p = 2$ and $n \in \{2, 4, 8, 16\}$.
Then, it can be easily checked that $-1$ is not an $n$th power modulo $p^{2k + 1}$.
By Theorem~\ref{thm:squares-cubes}, $R(S_2^2)$ is not dense in $\mathbb{Q}_p$ and, since $S_2^n \subseteq S_2^2$, we get that $R(S_2^n)$ is not dense in $\mathbb{Q}_p$.
Now assume $(n, p) \notin \{(2, 2), (4, 2), (8, 2), (16, 2)\}$.
By Theorem~\ref{thm:power-sums}, we have that $R(S_2^n)$ is dense in $\mathbb{Q}_p$ if and only if there exist integers $0 \leq x_1, x_2 < p^{2k + 1}$, not both divisible by $p$, such that $x_1^n + x_2^n$ is divisible by $p^{2k + 1}$.
It easy to see that this last condition is equivalent to the $-1$ being an $n$th power modulo $p^{2k + 1}$.
\end{proof}

In \cite[Problem~4.3]{MR3670202} it is asked about the denseness in $\mathbb{Q}_p$ of $R(S_m^4)$ and $R(S_m^5)$.
From Corollary~\ref{cor:minusone}, we have that $R(S_m^5)$ is dense in $\mathbb{Q}_p$ for all integers $m \geq 2$ and prime numbers $p$.
Regarding $R(S_m^4)$, the situation is more complicated.
Theorem~\ref{thm:power-sums}(d) already covers the case $p = 2$.
For $p > 2$ we have the following result.

\begin{thm}\label{thm:fourth-fifth}
For all prime numbers $p > 2$, we have:
\begin{enumerate}[\phantom{m}(a)]
\setlength{\itemsep}{0.5em}
\item $R(S_2^4)$ is dense in $\mathbb{Q}_p$ if and only if $p \equiv 1 \pmod 8$.
\item $R(S_3^4)$ is dense in $\mathbb{Q}_p$ if and only if $p \neq 5,29$.
\item $R(S_4^4)$ is dense in $\mathbb{Q}_p$ if and only if $p \neq 5$.
\item $R(S_m^4)$ is dense in $\mathbb{Q}_p$ for all integers $m \geq 5$.
\end{enumerate}
\end{thm}
\begin{proof}
By Corollary~\ref{cor:minusone}, $R(S_2^4)$ is dense in $\mathbb{Q}_p$ if and only if $-1$ is a fourth power modulo $p$.
In turn, this is well known to be equivalent to $p \equiv 1 \pmod 8$.
Hence, (a) is proved.
Substituting $a = -1$ into \eqref{equ:Xg}, the bound $\theta(n, b) \leq \gamma(n, b) + 1$ follows.
From \cite[Theorem~$3^\prime$]{MR0439787}, we have that $\gamma(4, p) = 2$ for all prime numbers $p > 41$.
Hence, $\theta(4, p) \leq 3$ for all prime numbers $p > 41$.
Then, a computation shows that $\theta(4, p) \leq 3$ for all prime numbers $p \neq 5, 29$.
Precisely, $\theta(4, 5) = 5$ and $\theta(4, 29) = 4$.
Now the claims (b), (c), and (d) follow from Theorem~\ref{thm:power-sums}.
\end{proof}

We leave the following general question to the readers.

\begin{que}
Given a prime number $p$ and a polynomial $f \in \mathbb{Z}[X]$, is there an effective criterion to establish if $R_f$ is dense in $\mathbb{Q}_p$?
What about multivariate polynomials?
\end{que}

\begin{rmk}\label{rmk:mistake}
In \cite[Theorem~4.2]{MR3670202} it is stated that $R(S_2^3)$ is not dense in $\mathbb{Q}_3$.
This is not correct, since $R(S_2^3)$ is dense in $\mathbb{Q}_3$ in light of Corollary~\ref{cor:minusone}.
The mistake in the proof of \cite[Theorems~4.2]{MR3670202} is when, at point (b2), it is asserted that: ``If $x/y \in R(S_2^3)$ is sufficiently close to $3$ in $\mathbb{Q}_3$, then $\nu_3(x) = \nu_3(y) + 1$. Without loss of generality, we may suppose that $\nu_3(x) = 1$ and $\nu_3(y) = 0$.''
This is not true, because if $y$ is the sum of two cubes, then there is no guarantee that $y / 3^{\nu_3(y)}$ is still the sum of two cubes.
For instance, if $y = 1^3 + 5^3$ then $y / 3^{\nu_3(y)} = 14$ is not the sum of two cubes.
\end{rmk}

\subsection*{Notation}

For each prime number $p$, let $\nu_p$ denote the usual $p$-adic valuation, with the convention $\nu_p(0) := +\infty$.
For integers $a$ and $m > 0$, we write $(a \bmod m)$ for the unique integer $r \in {]{-b/2}, b/2]}$ such that $a - r$ is divisible by $m$.

\section{Proof of Theorem~\ref{thm:dens-cond}}\label{sec:2}

We have to prove that for all $r \in \mathbb{Q}_p$ and $u > 0$ there exist $x_1, x_2 \in \mathbb{Z}^+$ such that $f(x_2) \neq 0$ and 
\begin{equation*}
\nu_p\!\left(\frac{f(x_1)}{f(x_2)} - r\right) > u .
\end{equation*}
Clearly, since $\mathbb{Q}_p^*$ is dense in $\mathbb{Q}_p$, it is enough to consider $r \neq 0$.
Furthermore, since $\mathbb{Z}^+$ is dense in $\mathbb{Z}_p$ and $f$ is continuous, we can assume, less restrictively, $x_1, x_2 \in \mathbb{Z}_p$.
By hypothesis, for $i = 1, 2$, we have $f(X) = (X - z_i)^{\mu_i}g_i(X)$, where $g_i : \mathbb{Z}_p \to \mathbb{Q}_p$ is an analytic function such that $g_i(z_i) \neq 0$.
Put $x_i := y_i p^{k_i} + z_i$, for $i=1,2$, where $y_1, y_2 \in \mathbb{Z}_p \setminus \{0\}$ and $k_1, k_2 \in \mathbb{Z}^+$ will be chosen later.
Without loss of generality, we can assume $\nu_p(g_1(z_1)) \leq \nu_p(g_2(z_2))$.
Thus, setting $G := g_2(z_2) / g_1(z_1)$, we have $G \in \mathbb{Z}_p \setminus \{0\}$.
Since $g_1, g_2$ are continuous, for sufficiently large $k_1, k_2$ we have
\begin{equation}\label{equ:G}
\nu_p\!\left(G \cdot\frac{g_1(x_1)}{g_2(x_2)} - 1\right) > u - \nu_p(r) ,
\end{equation}
In particular, it is implicit that $g(x_2) \neq 0$ and consequently $f(x_2) \neq 0$.
We fix $k_1, k_2$ such that
\begin{equation*}
k_1\mu_1 - k_2\mu_2 = \nu_p(r),
\end{equation*}
and \eqref{equ:G} holds.
This is possible thanks to the condition $\gcd(\mu_1 , \mu_2) = 1$.
Indeed, by B\'ezout's lemma, the quantity $k_1 \mu_1 - k_2 \mu_2$ can be equal to any integer with $k_1$ and $k_2$ arbitrarily large (if $k_1 \mu_1 - k_2 \mu_2 = a$, then $(k_1 + K\mu_2)\mu_1 - (k_2 + K\mu_1 )\mu_2 = a$, for any integer $K$).

Again by B\'ezout's lemma, there exist integers $h_1, h_2 \geq 0$ such that $h_1 \mu_1 - h_2 \mu_2 = 1$.
We set $y_i = s^{h_i}$, for $i=1,2$, where $s := p^{-\nu_p(r)} r G$.
Note that $y_1, y_2 \in \mathbb{Z}_p \setminus\{0\}$, as required.

Hence, we have
\begin{align*}
\frac{f(x_1)}{f(x_2)} &= \frac{(x_1 - z_1)^{\mu_1}}{(x_2 - z_2)^{\mu_2}} \cdot \frac{g_1(x_1)}{g_2(x_2)} = p^{k_1\mu_1 - k_2 \mu_2} \cdot \frac{y_1^{\mu_1}}{y_2^{\mu_2}}\cdot \frac{g_1(x_1)}{g_2(x_2)} \\
&= p^{\nu_p(r)} \cdot s^{h_1\mu_1 - h_2 \mu_2}\cdot \frac{g_1(x_1)}{g_2(x_2)} = p^{\nu_p(r)} \cdot s\cdot \frac{g_1(x_1)}{g_2(x_2)} = r G\cdot \frac{g_1(x_1)}{g_2(x_2)} ,
\end{align*}
so that, recalling \eqref{equ:G}, we get
\begin{equation*}
\nu_p\!\left(\frac{f(x_1)}{f(x_2)} - r\right) = \nu_p\!\left(r\left(G\cdot\frac{g_1(x_1)}{g_2(x_2)} - 1\right)\right) > u ,
\end{equation*}
as desired.

\section{Proof of Theorem~\ref{thm:power-sums}}\label{sec:3}

(a) Suppose that there exist integers $0 \leq x_1, \ldots, x_m < p^{2k + 1}$, not all divisible by $p$, such that $x_1^n + \dots + x_m^n$ is divisible by $p^{2k + 1}$.
Up to reordering $x_1, \dots, x_m$, we can assume that $p \nmid x_1$.
Put $f(X) = X^n + x_2^n + \cdots + x_m^n$, so that $f^\prime(X) = n X^{n - 1}$.
In particular, all the roots of $f$ are simple.
Since $p \nmid x_1$, we have
\begin{equation*}
\nu_p(f(x_1)) \geq 2k + 1 > 2k = 2\nu_p(f^\prime(x_1)) ,
\end{equation*}
so that, by Hensel's lemma~\cite[Ch.~4, Lemma~3.1]{CASSELS}, $f$ has a simple root in $\mathbb{Z}_p$.
Hence, by Corollary~\ref{cor:dens-cond}, $R_f$ is dense in $\mathbb{Q}_p$.
Clearly, $R_f \subseteq R(S_m^n)$, so that $R(S_m^n)$ is dense in $\mathbb{Q}_p$.

(b) Suppose that there are no integers $x_1, \dots, x_m$ as before, and that 
\begin{equation}\label{equ:excluded}
(n,p) \notin \{(2, 2), (4, 2), (8, 2), (16, 2)\} .
\end{equation}
We shall prove that $4k + 1 < n$.
For the sake of contradiction, suppose $4k + 1 \geq n$.
Since $n \geq 2$, we have $k \geq 1$.
Also, we have $4k + 1 \geq p^k$, which implies $p \leq 5$.
Now, taking into account \eqref{equ:excluded}, it can be readily checked that
\begin{equation*}
(n,p) \in \{(3, 3), (9, 3), (5, 5)\} .
\end{equation*}
But $3^3 \!\mid\! (1^3 + 8^3)$, $3^5 \!\mid\! (1^9 + 26^9)$, and $5^3 \!\mid\! (1^5 + 24^5)$, contradicting the nonexistence of $x_1, \dots, x_m$.

Let $y_1, \dots, y_m \geq 0$ be integers, not all equal to zero.
Put $\mu := \min\{\nu_p(y_i) : i = 1,\dots,m\}$, $I := \{i : \nu_p(y_i) = \mu\}$, and $J := \{1,\dots,m\} \setminus I$.
Also, put $z_i := y_i / p^\mu$ for $i \in I$, so that $z_i$ is an integer not divisible by $p$.
The nonexistence of $x_1, \dots, x_m$ implies that
\begin{equation}\label{equ:sumzi}
\nu_p\!\left(\sum_{i \in I} z_i^n \right) \leq 2k .
\end{equation}
Therefore, since $2k < n$, we have
\begin{equation*}
\nu_p\!\left(\sum_{i \in I} y_i^n \right) = \mu n + \nu_p\!\left(\sum_{i \in I} z_i^n \right) \leq \mu n + 2k < (\mu + 1) n \leq \nu_p\!\left(\sum_{j \in J} y_j^n \right) ,
\end{equation*}
and consequently
\begin{equation*}
\nu_p\!\left(y_1^n + \cdots + y_m^n \right) = \nu_p\!\left(\sum_{i \in I} y_i^n \right) = \mu n + \nu_p\!\left(\sum_{i \in I} z_i^n \right) ,
\end{equation*}
which in turn, by \eqref{equ:sumzi}, implies that
\begin{equation*}
(\nu_p\!\left(y_1^n + \cdots + y_m^n \right) \bmod n) \in \{0, \ldots,2k\} .
\end{equation*}
Thus, for each $a \in R(S_m^n) \setminus \{0\}$ we have
\begin{equation*}
(\nu_p(a) \bmod n) \in \{-2k,\ldots,2k\} ,
\end{equation*}
that is, the $p$-adic valuations of the nonzero elements of $R(S_m^n)$ belong to at most $4k + 1$ residue classes modulo $n$.
Since $4k + 1 < n$, at least one residue class modulo $n$ is missing and, a fortiori, $R(S_m^n)$ is not dense in $\mathbb{Q}_p$.

(c) The claim follows immediately from Theorem~\ref{thm:squares-cubes}.

From now on, assume $n = 2^k$, with $k \in \{2, 3, 4\}$.
Let $T_m^n$ be the topological closure of $S_m^n$ in $\mathbb{Q}_2$.
Clearly, we have
\begin{equation*}
T_m^n = \left\{x_1^n + \cdots + x_m^n : x_1, \ldots,x_m \in \mathbb{Z}_2 \right\} .
\end{equation*}
It is a standard exercise showing that the nonzero $n$th powers of $\mathbb{Z}_2^*$ are exactly the elements of the form $1 + 4ny$, with $y \in \mathbb{Z}_2$.
As a consequence, 
\begin{equation*}
T_1^n = \{2^{nv}(1 + 4ny) : v \in \mathbb{Z}_{\geq 0},\, y \in \mathbb{Z}_2\} \cup \{0\} .
\end{equation*}
Let $v_1, v_2 \geq 0$, $j \geq 1$ be integers and $y_1, y_2 \in \mathbb{Z}_2$.
If $v_1 = v_2$, then
\begin{equation*}
2^{nv_1}(j + 4ny_1) + 2^{nv_2}(1 + 4ny_2) = 2^{nv_1}(j + 1 + 4nz) ,
\end{equation*}
where $z := y_1 + y_2 \in \mathbb{Z}_2$.
If $v_1 < v_2$, then
\begin{equation*}
2^{nv_1}(j + 4ny_1) + 2^{nv_2}(1 + 4ny_2) = 2^{nv_1}(j + 4nz) ,
\end{equation*}
where $z := y_1 + 2^{n(v_2 - v_1)-k-2}(1 + 4ny_2) \in \mathbb{Z}_2$, since $n = 2^k \geq k + 2$.
If $v_1 > v_2$, then
\begin{equation*}
2^{nv_1}(j + 4ny_1) + 2^{nv_2}(1 + 4ny_2) = 2^{nv_2}(1 + 4nz) ,
\end{equation*}
where $z := 2^{n(v_1 - v_2)-k-2}(j + 4ny_1) + y_2 \in \mathbb{Z}_2$, again since $n \geq k + 2$.

Therefore, it follows easily by induction on $m$ that
\begin{equation}\label{equ:Tmn}
T_m^n = \{2^{nv}(j + 4ny) : v \in \mathbb{Z}_{\geq 0},\, j \in \{1,\dots,m\},\, y \in \mathbb{Z}_2\} \cup \{0\} .
\end{equation}

(d) On the one hand, using \eqref{equ:Tmn}, it can be checked quickly that $15 \notin R(T_{7}^4)$.
Hence, $R(S_{7}^4)$ is not dense in $\mathbb{Q}_2$.
On the other hand, we have
\begin{equation*}
2^{4v + r}(1 + 2y) = \frac{2^{4v}(8 + 16y)}{2^{4\cdot0}(2^{3-r} + 16\cdot 0)} \in R(T_8^4) ,
\end{equation*}
for all $v \in \mathbb{Z}_{\geq 0}$, $r \in \{0,1,2,3\}$, and $y \in \mathbb{Z}_2$.
Hence, $\mathbb{Z}_p \subseteq R(T_8^4)$ and, since $R(T_8^4)$ is closed by inversion, we get that $R(T_8^4) = \mathbb{Q}_p$.
Thus $R(S_8^4)$ is dense in $\mathbb{Q}_p$.

(e) On the one hand, by \eqref{equ:Tmn}, the $2$-adic valuation of each nonzero element of $T_{15}^8$ is congruent to $0$, $1$, $2$, or $3$ modulo $8$.
Hence, $R(T_{15}^8)$ contains no element with $2$-adic valuation equal to $4$, and consequently $R(S_{15}^8)$ is not dense in $\mathbb{Q}_2$.
On the other hand, we have
\begin{equation*}
2^{8v + r}(1 + 2y) = \frac{2^{8v}(16 + 32y)}{2^{8\cdot0}(2^{4-r} + 32\cdot 0)} \in R(T_{16}^8) 
\end{equation*}
and
\begin{equation*}
2^{8v + r + 4}(1 + 2y) = \frac{2^{8(v+1)}(2^r + 32\cdot0)}{2^{8\cdot0}(16 + 32\frac{-y}{1+2y})} \in R(T_{16}^8) 
\end{equation*}
for all $v \in \mathbb{Z}_{\geq 0}$, $r \in \{0,1,2,3,4\}$, and $y \in \mathbb{Z}_2$.
Hence, $\mathbb{Z}_p \subseteq R(T_{16}^8)$ and, since $R(T_{16}^8)$ is closed by inversion, we get that $R(T_{16}^8) = \mathbb{Q}_p$.
Thus $R(S_{16}^8)$ is dense in $\mathbb{Q}_p$.

(f) On the one hand, by \eqref{equ:Tmn}, the $2$-adic valuation of each nonzero element of $T_{63}^{16}$ is congruent to $0$, $1$, $2$, $3$, $4$, or $5$ modulo $16$.
Hence, $R(T_{63}^{16})$ contains no element with $2$-adic valuation equal to $6$, and consequently $R(S_{63}^{16})$ is not dense in $\mathbb{Q}_2$.
On the other hand, $2^9$ divides $5^{16} + 1^{16} + \cdots + 1^{16}$ ($63$ times $1^{16}$).
Hence, by point (a), we get that $R(S_{64}^{16})$ is dense in $\mathbb{Q}_2$.

\subsection*{Acknowledgments}

The authors thanks the anonymous referee for carefully reading the paper. 
N.~Murru and C.~Sanna are members of the INdAM group GNSAGA.

\end{document}